\documentclass{article} 
\usepackage{amsmath,amsfonts}
\usepackage{algorithmic}
\usepackage{array}
\usepackage[caption=false,font=normalsize,labelfont=sf,textfont=sf]{subfig}
\usepackage{textcomp}
\usepackage{stfloats}
\usepackage{url}
\usepackage{verbatim}
\usepackage{graphicx}
\def\BibTeX{{\rm B\kern-.05em{\sc i\kern-.025em b}\kern-.08em
    T\kern-.1667em\lower.7ex\hbox{E}\kern-.125emX}}
\usepackage{balance}
\usepackage[utf8]{inputenc}
\usepackage[T1]{fontenc}
\usepackage[rightcaption]{sidecap}

\usepackage[inner=30mm,outer=30mm,tmargin=30mm,bmargin=30mm]{geometry}

\usepackage{latexsym,amsfonts,amsmath,amssymb,epsfig,tabularx,amsthm,dsfont,mathrsfs}
\usepackage{graphicx}
\usepackage{enumerate}
\usepackage{bbm} 
\usepackage{tikz}
\usepackage{verbatim}
\usetikzlibrary{arrows}
\usetikzlibrary{plotmarks}

\usepackage[colorlinks=false]{hyperref} 
\hypersetup{linkbordercolor=red,
        linkcolor=black,
        citecolor=black,
        urlcolor=blue}
\usepackage{caption}

\usepackage{subfig}
				
\usepackage{cprotect} 

\usepackage{caption} 
\usepackage{multirow}
\usepackage{cprotect} 
\theoremstyle{plain}
\usepackage{dsfont} 

\usepackage{amsfonts,mathrsfs,amssymb}
\usepackage{graphicx}
\usepackage{epstopdf}
\usepackage{algorithmic}
 \ifpdf
   \DeclareGraphicsExtensions{.eps,.pdf,.png,.jpg}
 \else
   \DeclareGraphicsExtensions{.eps}
 \fi

\usepackage{todonotes}

\usepackage{enumerate} 

\usepackage{longtable}

\usepackage{bbm}
\usepackage{bibentry} 
\makeatletter\let\saved@bibitem\@bibitem\makeatother

\usepackage[colorlinks=false]{hyperref} 
\hypersetup{linkbordercolor=red,
        linkcolor=black,
        citecolor=black,
        urlcolor=blue}
\makeatletter\let\@bibitem\saved@bibitem\makeatother

\usepackage{tikz}
\usepackage{verbatim}
\usetikzlibrary{arrows}
\usetikzlibrary{plotmarks}
\usetikzlibrary{shapes.geometric}
\usetikzlibrary{positioning}
\usetikzlibrary{patterns}
\usetikzlibrary{calc,math}
\usetikzlibrary{fadings}

\tikzfading[name=fade out,
inner color=transparent!0,
outer color=transparent!100]

\usepackage{listings} 
\usepackage{cleveref}

\usepackage{xifthen}

\usepackage{bm}

\usepackage{mathtools}

\usepackage{datetime}
\newdate{date}{\today} 
\date{\displaydate{data}}

\newcommand{\diff}{\mathup{d}} 

\DeclareMathOperator{\supp}{supp}

\DeclareMathOperator{\diag}{diag}

\DeclareMathAlphabet{\mathup}{OT1}{\familydefault}{m}{n}

\newcommand{\sep}{\mathup{sep}}

\newcommand{\ii}{\mathrm{i}}
\newcommand{\e}{\mbox{e}}
\newcommand{\eip}[1]{{\e}^{ 2\pi{\ii} #1}}
\newcommand{\eim}[1]{{\e}^{-2\pi{\ii} #1}}

\newcommand{\N}{\ensuremath{\mathbb{N}}}

\newcommand{\T}{\ensuremath{\mathbb{T}}}

\newcommand{\Z}{\ensuremath{\mathbb{Z}}}

\newcommand{\R}{\ensuremath{\mathbb{R}}}
\newcommand{\C}{\ensuremath{\mathbb{C}}}

\newcommand{\norm}[2][]{\ifthenelse{\isempty{#1}}%
  {\left\Vert #2\right\Vert}%
  {\left\Vert #2\right\Vert_{#1}}}

\newcommand{\dirich}[1]{\ifthenelse{\isempty{#1}}
	{d_n}
	{d_n\left(#1\right)} }		
\newcommand{\dirichd}[1]{\ifthenelse{\isempty{#1}}
	{d_n'}
	{d_n'\left(#1\right)} }	
\newcommand{\dirichdd}[1]{\ifthenelse{\isempty{#1}}
	{d_n''}
	{d_n''\left(#1\right)} }	
\newcommand{\dirichm}[1]{\ifthenelse{\isempty{#1}}
	{\tilde d_n}
	{\tilde d_n\left(#1\right)} }

\renewcommand{\mathbf}[1]{\ensuremath{\boldsymbol{#1}}}

\definecolor{mycolor}{rgb}{0,0,0}

\newtheorem{thm}{Theorem}[section]
\newtheorem{lemma}[thm]{Lemma}
\newtheorem{corollary}[thm]{Corollary}
\newtheorem{proposition}[thm]{Proposition}
\theoremstyle{definition}

\newtheorem{definition}[thm]{Definition}

\crefname{lemma}{Lemma}{Lemmata}
\crefname{proposition}{Proposition}{Propositions}
\crefname{definition}{Definition}{Definitions}
\crefname{theorem}{Theorem}{Theorems}
\crefname{thm}{Theorem}{Theorems}
\crefname{corollary}{Corollary}{Corollaries}
\crefname{equation}{}{}
\crefname{remark}{Remark}{Remarks}
\crefname{algorithm}{Algorithm}{Algorithms}
\crefname{chapter}{Chapter}{Chapters}
\crefname{section}{Section}{Sections}
\crefname{table}{Table}{Tables}
\crefname{figure}{Figure}{Figures}
\crefname{example}{Example}{Examples}
\crefname{appendix}{Appendix}{Appendices}
\crefname{subsection}{Subsection}{Subsections} 

\begin{document}
\title{Analysis of the sparse super resolution limit using the Cram\'{e}r-Rao lower bound}
\author{Mathias Hockmann\thanks{
The author gratefully acknowledges support by the DFG within the Collaborative Research Center 944 
``Physiology and dynamics of cellular microcompartments’’.\\
Institute of Mathematics at Osnabr\"uck University, Germany, \\
email: \texttt{mahockmann@uos.de}}
}
\markboth{to be filled by the editor}
{to be filled by the editor}

\maketitle

\begin{abstract}
Already since the work by Abbe and Rayleigh the difficulty of \textit{super resolution} where one wants to recover a collection of point sources from low-resolved microscopy measurements is thought to be dependent on whether the distance between the sources is below or above a certain \textit{resolution} or \textit{diffraction limit}. Even though there has been a number of approaches to define this limit more rigorously, there is still a gap between the situation where the task is known to be hard and scenarios where the task is provably simpler. For instance, an interesting approach for the univariate case using the size of the \textit{Cram\'{e}r-Rao lower bound} was introduced in a recent work by Ferreira Da Costa and Mitra. In this paper, we prove their conjecture on the transition point between good and worse tractability of super resolution and extend it to higher dimensions. Specifically, the bivariate statistical analysis allows to link the findings by the Cram\'{e}r-Rao lower bound to the classical Rayleigh limit.       
\end{abstract}





	
	

\section{Introduction} 

Super resolution (SR) as the task to recover a signal from bandlimited information is well-studied and has applications in many inverse problems including microscopy, e.g.\,see \cite{Laville_21,Ovesny_14,Rust_2006,Ingerman19}. Classically, the signal is assumed to be a convolution of a sparse measure of interest $\mu$ and a bandlimited \textit{point spread function} (PSF) such that we model this as an ideal low pass filter yielding an estimate for the low order Fourier coefficients of $\mu$. Given these estimates on the Fourier coefficients, the actual SR-problem is then to find the parameters of the discrete measure $\mu$ consisting of node positions $t$ and weights $\alpha_t$ in the representation
\begin{align*}
    \mu=\sum_{t\in Y} \alpha_t \delta_t
\end{align*}
where $\delta_t$ is the \textit{Dirac measure} at $t$ and $Y$ some finite set, cf.\,\cite{Laville_21}. 

While there exist many algorithms for this problem using variational techniques \cite{Fernanadez_16,candes13,Laville_21}, subspace methods beginning with Prony \cite{Prony_1795} or more recently machine learning approaches \cite{Nehme_18,Nehme_20,Speiser_20}, there has been a lot of work on stability analysis of specific algorithms e.g.\,cf.\,\cite{Fernanadez_16,candes13,Li_20,liao14,Aubel_16,Potts_17,sahnoun17,Fan2022}. In order to select an optimal algorithm it is then natural to ask how much instability of an algorithm is caused by the problem itself. In other words, one wants to understand the \textit{condition} of SR and we refer to \cite{Liu_21_line,Batenkov_2020,Chen_21,Eftekhari_21a,DaCosta_22} for some results from the literature. Among these, we highlight the recent result by Ferreira Da Costa and Mitra \cite{DaCosta_22} which estimates the condition of SR by the \textit{Cram\'{e}r-Rao} (CR) lower bound and allows to conclude well-conditionedness in the univariate case if the distance of support nodes $q$ and the bandlimit parameter $n$ admit $(2n+1)\cdot q >3.54$. The introduction of a new multivariate Beurling-Selberg type minorant by the author of this work in \cite{Ho_23b} allows to sharpen this result on the resolution limit in an optimal way and to extend it to higher dimensions.

\subsection{Contributions and outline of the paper}
We prove the conjecture from \cite[p.\,1741]{DaCosta_22} by showing that univariate SR is well-conditioned in the sense of the size of its Cram\'{e}r-Rao lower bound if the separation fulfills the sharp bound $n\cdot q>1$. Together with a multivariate extension, this is formulated in the main result given by \cref{Cor_condition_CR}. Its most important implication is that it strongly justifies a resolution limit, known historically as the \textit{Rayleigh limit} \cite{Rayleigh_1879}, as the point where SR transitions from well- to ill-conditionedness.\par
\Cref{sec_main_results} containing the main results of this paper is divided into three subsections. At first, we compute the Cram\'{e}r-Rao lower bound of the super resolution problem and define our notion of its condition by the size of the smallest singular value of the \textit{Fisher information matrix}. In order to prove the main technical result \cref{Prop_condition_block_Vandermonde}, we recapitulate properties of our minorant function as introduced in \cite{Ho_23b}. Finally, the proof of \cref{Prop_condition_block_Vandermonde} is given in \cref{Subsec:proof}.

\subsection{Notation}

For a matrix $A$ (or a vector), we denote its complex conjugate by $A^*$ and its transpose by $A^\top$. If $A,B$ are two Hermitian matrices such that $A-B$ is positive semidefinite, we write $A\succeq B$. Additionally, a diagonal matrix $A$ with vector $v$ on its main diagonal and zero elsewhere is rewritten as $A=\diag(v)$. The smallest singular value of $A$ or its smallest eigenvalue (if $A$ is Hermitian) is denoted by $\sigma_{\min}(A)$ or $\lambda_{\min}(A)$ respectively. The norm $\|\cdot \|_2$ is the standard Euclidean norm for vectors and the corresponding operator norm for matrices. \par
Apart from this notation for objects from linear algebra, the torus as the periodic unit interval is $\T:=\R/\Z$ and $|Y|$ for a finite set $Y$ means the cardinality of the set. For a finite set $Y\subset \T^d$, its (minimal) separation is defined as
\begin{align*}
    q=\sep\,Y:=\min_{t,t'\in Y} \min_{\ell\in\Z^d} \|t-t'+\ell\|_2.
\end{align*}
Finally, $j_{\nu,1}$ stands for the first positive zero of the Bessel function of the first kind with order $\nu$ which is denoted by $J_{\nu}$, cf.\,\cite{Watson_1944}.

\section{Main results} \label{sec_main_results}

\subsection{Cram\'{e}r-Rao lower bound}

The \textit{Cram\'{e}r-Rao (CR) lower bound} estimates that the covariance of each unbiased estimator $\hat\theta$ for a vector of parameters $\theta$ can be bounded from below by the inverse of the \textit{Fisher information matrix} $J(\theta)$. We summarise known results about the CR lower bound and the Fisher information matrix (FIM) in the following theorem.
\begin{thm}(CR and Fisher information, cf.\,\cite[p.\,6424]{Pakrooh_15}) \label{thm_Fisher}
Assume that a random vector $y\in\C^m$ has probability density function $f(y,\theta)$ depending on some unknown, deterministic parameter $\theta\in \R^{l}$ for some $l\in\N$. Then, the \textit{Fisher information matrix} defined as\footnote{We emphasise that the expectation is computed with respect to the random $y$ by the subscript $\mathbbm{E}_y$.} 
\begin{align*}
    \scalebox{0.99}{$J(\theta)=\mathbbm{E}_y\left[\left(\frac{\partial \log f(y,\theta)}{\partial \theta}\right)\left(\frac{\partial \log f(y,\theta)}{\partial \theta}\right)^*\right] \in \C^{l\times l}$}
\end{align*}
satisfies the covariance estimate $$\mathbbm{E}_y\left[(\hat\theta(y)-\theta)(\hat\theta(y)-\theta)^*\right]\succeq J(\theta)^{-1}$$ for any unbiased estimator $\hat\theta$. If $y$ follows a multivariate complex normal distribution with mean $x(\theta)\in\C^m$ and diagonal covariance matrix $\delta^2 I$ for some $\delta>0$, i.e. $y\sim \mathcal{C}\mathcal{N}(x(\theta),\delta^2 I) $, then the Fisher information matrix can be calculated as 
\begin{align*}
    J(\theta)&=\delta^{-2} G^* G, \\
    \text{where } G&=\left[\frac{\partial x(\theta)}{\partial \theta_1}, \dots, \frac{\partial x(\theta)}{\partial \theta_l}\right] \in\C^{m\times l}.
\end{align*}
\end{thm}
From a theoretical point of view, this theorem allows to derive a minimal covariance of any algorithm recovering the $\theta$ from measurements $y$. Hence, this can be seen as a lower bound on the condition of the problem itself. This theory was therefore used in the context of univariate super resolution in \cite{DaCosta_22} by assuming that the measured moments are
\begin{align}
    \hat{\mu}(k)=\sum_{t\in Y} \alpha_t \eim{t\cdot k} + \hat\rho(k) \label{eq_noise_normal}
\end{align}
for some normally distributed noise vector $\hat\rho\sim\mathcal{C}\mathcal{N}(0,\delta^2 I)$ and this can be directly generalised from the univariate case $k\in\{-n,\dots,n\}$ to the higher dimensional case $$\mathcal{I}:=\{k\in\Z^d: \|k\|_2\leq n\}, \quad d\geq 1,$$ for some bandlimit parameter $n>0$. Here, the vector of unknown parameters $\theta$ is 
\begin{align*}
\theta =&\,(\alpha,Y)^\top \\
:=& \left[\left(\alpha_t\right)_{t\in Y}, \left(t_1\right)_{t\in Y} \cdots \left(t_d\right)_{t\in Y} \right]^\top\in \C^{|Y|(d+1)}
\end{align*}
and the measurements are $$y=\left(\hat{\mu}(k)\right)_{\{k\in\Z^d: \|k\|_2\leq n\}}\in \C^{|\mathcal{I}|}.$$ Based on this model, we can compute the Fisher information matrix as follows.

\begin{corollary}
\label{Lem_Fisher}
If the moments satisfy the noise model \eqref{eq_noise_normal}, we have the factorisation $J(\alpha,Y)=\delta^{-2} G^* G$ of the Fisher information matrix where $$G= \left( \mathscr{A} ,\tilde{\mathscr{A}}_{1}, \cdots, \tilde{\mathscr{A}}_{d} \right) D_{\alpha}$$ with a \textit{Vandermonde matrix}
\begin{align*}
    \mathscr{A} = \left(\eim{tk}\right)_{k\in \{k\in\Z^d: \|k\|_2\leq n\},\,t\in Y} \in \C^{|\mathcal{I}|\times |Y|},
\end{align*}
matrices $\tilde{\mathscr{A}}_{s}$, $s=1,\dots,d$, with
\begin{align*}
    \scalebox{0.9}{$\tilde{\mathscr{A}}_{s}=-2\pi\ii\left( k_s \eim{tk}\right)_{k\in \{k\in\Z^d: \|k\|_2\leq n\},\,t\in Y} \in \C^{|\mathcal{I}|\times |Y|}$}
\end{align*}
and the diagonal matrix
\begin{align*}
     D_{\alpha} 
    :=&\,\scalebox{0.86}{$\diag(\underbrace{1,\dots,1}_{|Y|\text{ times}},\underbrace{\alpha_{t_1}, \dots, \alpha_{t_{|Y|}},\dots, \alpha_{t_{|1|}}, \dots, \alpha_{t_{|Y|}}}_{\text{repeat weight vector } d\text{ times}})$} \\
    \in&\,\C^{|Y|(d+1)\times |Y|(d+1)}.
\end{align*}
The matrices $\tilde{\mathscr{A}}_{s}$ can be seen as a variant of a \textit{confluent Vandermonde matrix} (see \cite{Gautschi_62}).
\end{corollary} 

\begin{proof}
The univariate case $d=1$ was given in \cite{DaCosta_22} and the higher dimensional result follows from \cref{thm_Fisher} by differentiating \eqref{eq_noise_normal} with respect to the parameters. The derivative with respect to the weights gives the Vandermonde matrix $\mathscr{A}$ while the partial derivative with respect to the $s$th component of every node gives the confluent Vandermonde matrix $\tilde{\mathscr{A}}_{s}$.
\end{proof}
As the inverse of the Fisher information matrix is then a lower bound for the covariance of each unbiased estimator, it is natural to define the condition of super resolution through the size of $J(\theta)^{-1}$ and the problem is considered to be ill-conditioned if $\|J(\theta)^{-1}\|$ becomes large or equivalently $\lambda_{\min}(J(\theta))$ is very small. Hence, one is interested to find lower bounds on the minimal eigenvalue of $J(\theta)$ in order to establish well-conditionedness. This approach was introduced in \cite{DaCosta_22} by defining the transition between good and ill-conditionedness as follows.\footnote{In \cite{DaCosta_22}, the authors use the term stability instead of condition. Since we distinguish between the condition of a problem and the stability of an algorithm (cf.\,\cite{Buergisser_13}), we proceed by using the term condition.}
\begin{definition}(Condition via CR, generalisation of \cite[Def.\,1]{DaCosta_22}) \label{Def_condition_CR}
SR is said to be \textit{well-conditioned} for some $\tilde q>0$ if for all $n$ and parameter configurations with separation $n \cdot \sep\,Y \geq \tilde q$ and some minimal absolute value of all weights $\alpha_{\min}>0$ there exists a constant $c_{\tilde{q},\alpha_{\min}}$ independent of $n$ such that
\begin{align*}
    n^{-d} \lambda_{\min}(J(\theta)) \geq \delta^{-2} c_{\tilde{q},\alpha_{\min}}.
\end{align*}
\end{definition}
Due to \cref{Lem_Fisher}, one directly finds under the natural assumption $|\mathcal{I}|\geq (d+1)|Y|$ that
\begin{align}
    \lambda_{\min}(J(\theta)) &= \delta^{-2} \sigma_{\min}^2(G) \nonumber\\
    &\geq \scalebox{0.87}{$\frac{\min(1,\alpha_{\min}^2)}{\delta^2} \sigma_{\min}^2 \left(\mathscr{A} ,\tilde{\mathscr{A}}_{1}, \cdots, \tilde{\mathscr{A}}_{d} \right)$} \label{eq:smallest_sing_arises}
\end{align}
with equality if all weights are equal to one.\footnote{In particular, we remark at this point that this analysis separates the dependency of the condition on the weights from the influence of the nodes. Moreover, the assumption $|\mathcal{I}|\geq (d+1)|Y|$ is natural because one should have an overdetermined problem with more data than unknown parameters.}
Consequently, the problems boils down to an estimate on the smallest singular value of a block matrix where each block consists of a Vandermonde matrix or of a confluent Vandermonde matrix. While there have been many attempts to analyse the smallest singular value of Vandermonde matrices, see \cite{Nagel_20} for an overview, Ferreira Da Costa and Mitra \cite{DaCosta_22} observed already for the one dimensional case that this can be done by using a so-called \textit{admissible function}, cf.\,\cref{Subsec:admissible}. Whereas they utilised a variant of the Beurling-Selberg minorant for this, we apply the function $\psi$ from \cref{lemma_psi_highD_ball} having optimally small support.

\begin{proposition}(Conditioning of partially confluent block Vandermonde matrix, generalisation of \cite[Prop.\,6]{DaCosta_22}) Assume that $n$ and the separation $q$ satisfy $nq=\sqrt{1+\tau}\frac{j_{d/2}}{\pi}$ for some $\tau>0$. Then, we have
\begin{align*}
    \sigma_{\min}^2 \left(\mathscr{A} ,\tilde{\mathscr{A}}_{1}, \cdots, \tilde{\mathscr{A}}_{d} \right) \geq c_{d,\tau} n^d
\end{align*}
for some constant $c_{d,\tau}>0$. \label{Prop_condition_block_Vandermonde}
\end{proposition}
\begin{proof}
    See \cref{Subsec:proof}.
\end{proof}

As a corollary of \cref{Prop_condition_block_Vandermonde}, we obtain an argument for using the Rayleigh limit $\frac{j_{d/2,1}}{\pi n}$ in order to describe the condition of super resolution. 

\begin{corollary} \label{Cor_condition_CR}
Let $d\in\N$. For all $\tilde q>\frac{j_{d/2,1}}{\pi}$, the super resolution problem is well conditioned in the sense of \cref{Def_condition_CR}.
\end{corollary}
\begin{proof}
This follows directly from \cref{Def_condition_CR}, \cref{Prop_condition_block_Vandermonde} and \eqref{eq:smallest_sing_arises}.
\end{proof}

In the univariate case, the sufficient condition from \cref{Cor_condition_CR} for well-conditioned\-ness reads $qn=\tilde q>1$ and this was already conjectured in \cite{DaCosta_22} where this conjecture was formulated in an asymptotically equivalent way as $q(2n+1)>2$. Moreover, \cite[Fig.\,2]{DaCosta_22} gives at least numerical evidence that $\tilde q>1$ is also necessary in the univariate situation. An approach to make this more precise by estimating the smallest singular value of $\left(\mathscr{A} ,\tilde{\mathscr{A}}_{1}, \cdots, \tilde{\mathscr{A}}_{d} \right)$ from above can be done by using results on $\sigma_{\min}(\mathscr{A})$. In fact, choosing $u=\left(u_0^\top, 0, \dots, 0\right)^\top\in\C^{|Y|(d+1)}$ where $u_0$ is the normalised right singular vector corresponding to the smallest singular value of $\mathscr{A}$ and $\alpha=\left(1,\dots,1\right)^\top\in\C^{|Y|}$ leads to
\begin{align}
    \delta^2\lambda_{\min}(J(\alpha,Y)) &= \sigma_{\min}^2 \left(\mathscr{A} ,\tilde{\mathscr{A}}_{1}, \cdots, \tilde{\mathscr{A}}_{d} \right) \nonumber\\
    &\leq \|\left(\mathscr{A} ,\tilde{\mathscr{A}}_{1}, \cdots, \tilde{\mathscr{A}}_{d} \right) u \|_2^2 \nonumber\\
    &=  \sigma_{\min}^2(\mathscr{A}). \label{eq_upper_bound_FIM}
\end{align}

\begin{figure}[ht]
    \centering
    \includegraphics[width=0.6\columnwidth]{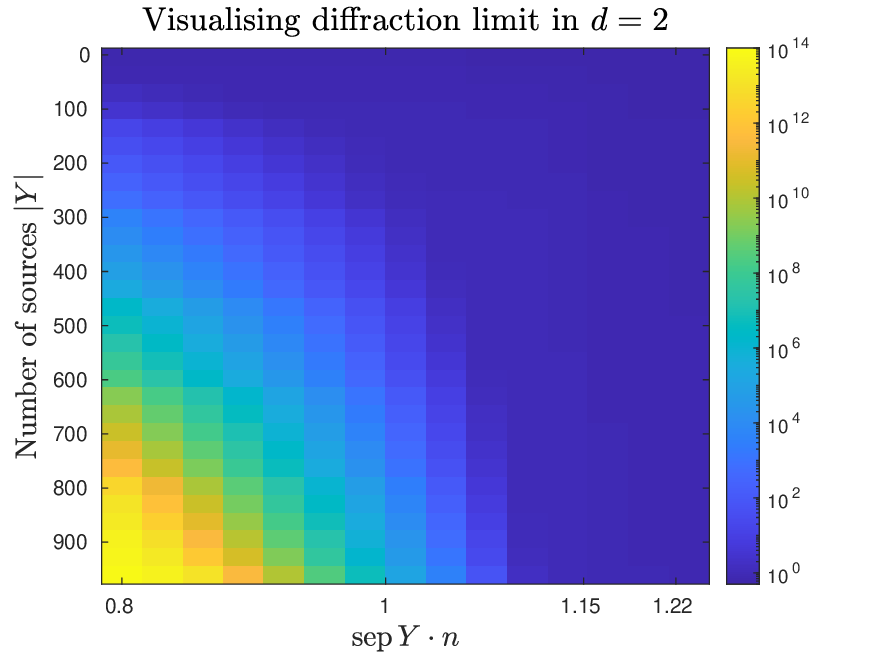}
    \caption[Bivariate diffraction limit]{Visualisation of the bivariate diffraction limit. We place the node set $Y$ on (a subset of) the two hexagonal lattices from \cite[Lem.\,2.1]{Chen_21} while varying the separation $\sep\,Y$ and the number of nodes $|Y|$ for fixed $n=40$. For each selection of $\sep\,Y$ and $|Y|$ we compute $n\cdot \sigma_{\min} \left(\mathscr{A} ,\tilde{\mathscr{A}}_{1}, \cdots, \tilde{\mathscr{A}}_{d} \right)^{-1}$ as a proxy for the condition of the super resolution problem.}
    \label{fig:condition}
\end{figure} 
 Even if the smallest singular values of Vandermonde matrices are well-studied, e.g. see \cite{Nagel_20} and the references therein, upper bounds on the smallest singular value for the case of ill-separated nodes in higher dimensions are difficult in general (cf.\,\cite[Subsec.\,3.4.4]{Nagel_20}). Nevertheless, the analysis from \eqref{eq_upper_bound_FIM} together with \cite[Thm.\,3.1]{Moitra_15} and \cite{Chen_21} shows that the super resolution problem cannot be well conditioned in the sense of \cref{Def_condition_CR} for $\tilde q<1$ and $\tilde{q}<\sqrt{\frac{4}{3}}$ in $d=1$ or $d=2$ respectively. 

The formulation of the condition in terms of singular values of certain matrices allows to compute this condition for visualisation in \cref{fig:condition}.\footnote{See \url{https://github.com/MHockmann/Dissertation} for the implementation of this computation.} In this numerical example, we see that the proxy $n\cdot \sigma_{\min} \left(\mathscr{A} ,\tilde{\mathscr{A}}_{1}, \cdots, \tilde{\mathscr{A}}_{d} \right)^{-1}$ for the condition number of super resolution can become large if $\sep\,Y\cdot n <\sqrt{\frac43}<\frac{j_{1,1}}{\pi}\approx 1.22$. In particular, the upper bound $\frac{j_{1,1}}{\pi}\approx 1.22$ is an improvement compared to \cite{Chen_20} where $\frac{2j_{0,1}}{\pi}\approx 1.53$
was used to guarantee well-conditionedness of SR. 

\subsection{Admissible functions} \label{Subsec:admissible}

We want to use a function $\psi$ with various properties to apply Poisson's summation formula in order to relate Fourier coefficients of a discrete measure $\mu$ to its parameters in real space. As we need a minorant in Fourier domain, we are interested in functions $\psi$ such that their Fourier transform $\hat{\psi}$ is a minorant to the indicator function of the Euclidean unit ball.\footnote{The Fourier transform of an integrable function $\psi$ is defined as $\hat\psi(v)=\int_{\R} \psi(x) \eim{vx} \diff x$.} Together with various other assumptions, we call such a function \textit{admissible}. Beyond the condition $\psi(0)>0$ used in \cite{Kunis_17}, we additionally require similar to \cite{Diederichs_18,Diederichs_19} that this is the global maximum.

\begin{definition}[Admissible function] \label{Def_admissible}
Let $d\in \N$ and $\psi: \R^d\to\R$ be a function $\psi\in L^1(\R^d)$ which
\begin{enumerate}
    \item [(i)] is continuous with compact support, i.e.\,$\psi\in C_c(\R^d)$,
    \item [(ii)] attains its global maximum $\psi(0)>0$ in the origin allowing to find $c_d>0$ such that the bound 
\begin{align*}
    \psi(0)-\psi(x)\ge c_d \|x\|_2^2
\end{align*}
for any $x\in\supp \psi$ holds,
     \item [(iii)] and satisfies $\hat\psi(v)\in\R$ for all $v\in\R^d$ with sign
    \begin{align} \label{eq:minorant}
     \hat\psi(v)
     \begin{cases}
      \ge 0 & \|v\|_2\le 1,\\
      \le 0 & \|v\|_2 > 1.
     \end{cases}
 \end{align}
\end{enumerate}
Then, we call a function $\psi$ fulfilling (i)-(iii) \emph{admissible}.
\end{definition}
 Classical results for functions satisfying (i) and (iii) can be found in \cite{Aubel_16,Vaaler_1985}. Additionally, a univariate function from \cite{Diederichs_18} also meets condition (ii). Based on the idea from \cite{Komornik_05}, we found the following admissible functions in the general multivariate case in \cite{Ho_23b}. For illustration, we display the function and its Fourier transform in \cref{fig:radial_psi}.

\begin{figure}[ht]
    \centering
    \subfloat[]{\includegraphics[width=0.4\columnwidth]{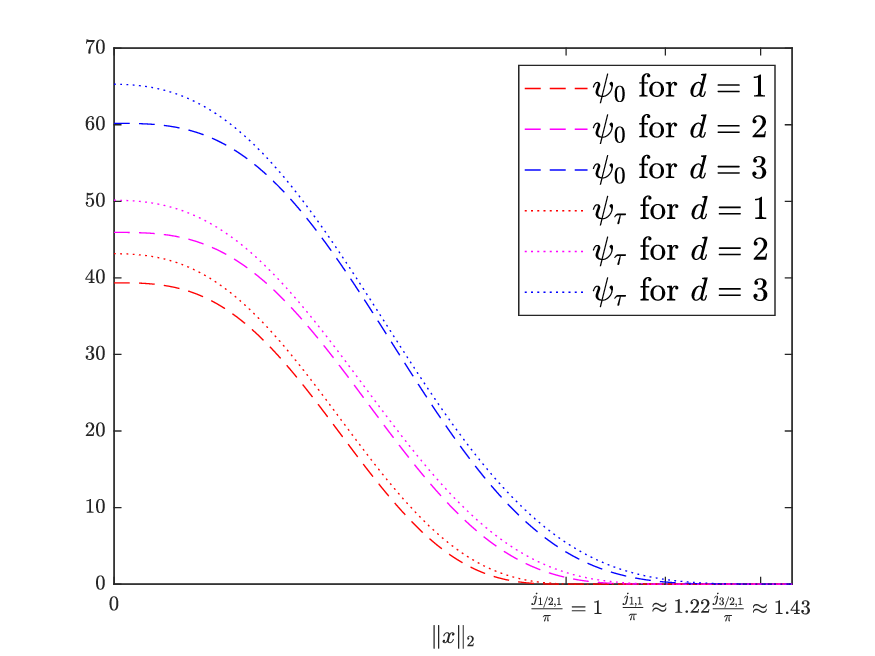}}
    \qquad
    \subfloat[]{\includegraphics[width=0.4\columnwidth]{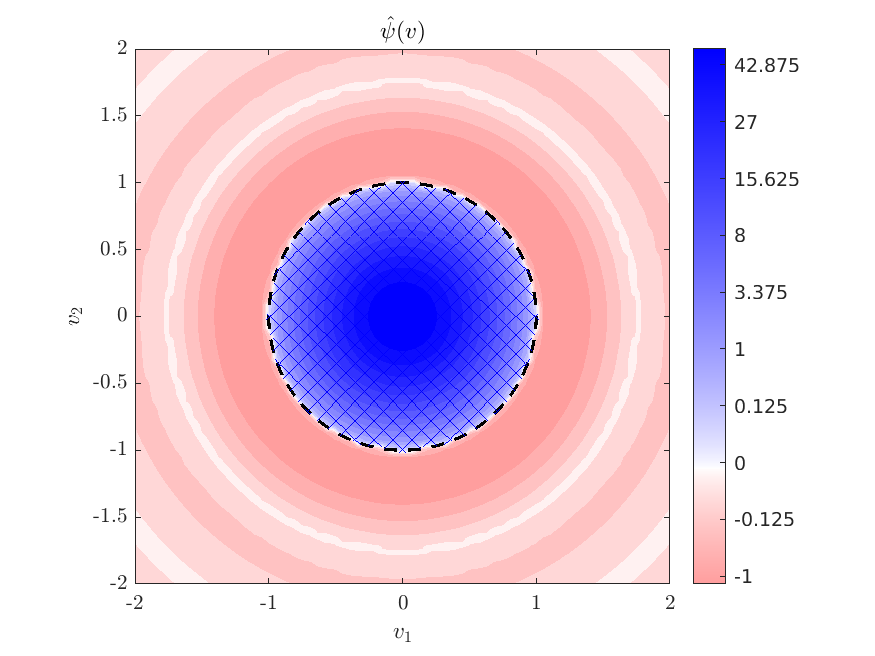}}
    \caption[Admissible functions]{Admissible function $\psi_{\tau}(x)$ for $\tau=0.1$ and non-admissible $\psi_0(x)$ for $d=1,2,3$ as a function of $\|x\|_2$ (a). One can at least imagine from its graph that the second derivative of $\psi_0$ at $x=0$ vanishes such that it is not admissible according to our definition as it does not fulfill (ii). Additionally, we display $\hat\psi_{0.1}$ for $d=2$ and highlight its radial dependency as well as the change of the sign at $\|v\|_2=1$ (b).}
    \label{fig:radial_psi}
\end{figure}

\begin{lemma}(Support on a ball, \cite[Lem.\,3.1]{Ho_23b}) \label{lemma_psi_highD_ball}
For $d\geq 1$ we define $\varphi\colon\R^d\to\R_{\geq 0}$,
\begin{align*}
    \varphi(x) = \scalebox{0.76}{$\left(1-\left(\frac{j_{d/2,1}}{2\pi\|x\|_2}\right)^{d/2-1} \frac{J_{d/2-1}(2\pi\|x\|_2)}{J_{d/2-1}(j_{d/2,1})}\right) \cdot \mathbbm{1}_{\frac{j_{d/2,1}}{2\pi}}(\|x\|_2) $}
\end{align*}
where $J_{\nu}$ denotes the $\nu$th Bessel function and $\mathbbm{1}$ the indicator function. Moreover, let $\triangle=\sum_{s=1}^d \frac{\partial^2}{\partial x_s^2}$ be the Laplace operator and for $\tau\geq 0$ the function $\psi_{\tau}\colon\R^d\to\R_{\geq 0}$,
\begin{align*}
    \psi_{\tau}(x)=\scalebox{0.86}{$\left(\frac{1}{\sqrt{1+\tau}}\right)^d\left[4\pi^2(1+\tau)+\triangle\right] (\varphi*\varphi)\left(\frac{x}{\sqrt{1+\tau}}\right).$}
\end{align*}
Then, $\psi_{\tau}$ with $\tau>0$ is admissible, its support satisfies $\supp \psi_{\tau}=B_{q_{\tau}}(0)$ with $$q_{\tau}:=\sqrt{1+\tau}\frac{j_{d/2,1}}{\pi},$$ and there is a constant $c_d>0$ depending only on $d$ that allows the estimate
\begin{align} \label{eq:low_bound_radial}
    \psi_{\tau}(0)-\psi_{\tau}(x) &\geq c_d \tau  (1+\tau)^{-d/2-1} \|x\|_2^2
\end{align}
for all $x\in\supp\psi_{\tau}$.
\end{lemma}

\subsection{Proof of \texorpdfstring{\cref{Prop_condition_block_Vandermonde}}{}} \label{Subsec:proof}
The proof uses the following lemma.
\begin{lemma}[Evaluating derivatives at zero] \label{Lem_derivatives_eval}
Let $\psi: \R^d\to \R$ be a radial function, i.e.\,$\psi(x)=h(\|x\|_2)$ for some univariate function $h$. Assume that $\psi$ and $h$ are twice continuously differentiable and that $\psi$ is maximal in zero. Then, we have $\left(\frac{\partial \psi}{\partial x_s}\right) (0)=0$ for all $s=1,\dots,d$. Moreover, one can find
\begin{align*}
 d \cdot \left(\frac{\partial^2 \psi}{\partial x_s^2} \right) (0)&= \triangle \psi (0)\quad \text{and} \\
 \left(\frac{\partial^2 \psi}{\partial x_s \partial x_{s'}}\right) (0)&=0
\end{align*}
for any $s,s'\in\{1,\dots,d\}, s\neq s'$.
\end{lemma}
\begin{proof}
The vanishing gradient follows directly from the extremum in zero. For the second derivatives one can calculate
\begin{align*}
      \scalebox{0.78}{$\left(\frac{\partial^2 \psi}{\partial x_s \partial x_{s'}}\right)(x)=\frac{h'(\|x\|_2)}{\|x\|_2} \delta_{s,s'} + \left((h''(\|x\|_2)-\frac{h'(\|x\|_2)}{\|x\|_2}\right) \frac{x_s x_{s'}}{\|x\|_2^2}.$}
\end{align*}
Because $h'(\|x\|_2)=h'(0)+h''(\|x\|_2)\|x\|_2+o(\|x\|_2)$ as $\|x\|_2\to 0$ and $\frac{|x_s x_{s'}|}{\|x\|_2^2}\leq 1$, the second part vanishes in zero. This yields that the mixed derivatives vanish in zero. Finally, the first term is independent of $s$ if $s=s'$. This gives the remaining part of the statement.
\end{proof}

We can then return to the proof of \cref{Prop_condition_block_Vandermonde}.

\begin{proof}[Proof of \cref{Prop_condition_block_Vandermonde}]
We follow the idea of the proof of \cite[Prop.\,6]{DaCosta_22} and define $$\psi_{\tau,n}(x):=n^d \psi_{\tau}(n\cdot x)$$ with $\psi_{\tau}$ from \cref{lemma_psi_highD_ball} having compact support in $B_q(0)$. By the variational representation of singular values, see \cite[Thm.\,7.3.8]{Horn_13}, we have to find a lower bound on the expression  $\|\left(\mathscr{A} ,\tilde{\mathscr{A}}_{1}, \cdots, \tilde{\mathscr{A}}_{d} \right) u\|_2$ for any normalised vector $u$ with block structure $u=\left(u_0^\top, u_1^\top, \dots, u_{d}^\top\right)^\top\in\C^{|Y|(d+1)}$ where $u_s\in\C^{|Y|}, s=1,\dots,d$. We set
\begin{align*}
    \hat\mu_0(k)&:=\sum_{t\in Y} (u_0)_t\eim{tk} \quad \text{and} \\
    \hat\mu_s(k)&:= -\sum_{t\in Y} 2\pi\ii k_s (u_{s})_t\eim{tk}
\end{align*}
for $s=1,\dots,d$ and $k\in \mathcal{I}$. Now we can compute
\begin{align*}
    &\hspace{0.05cm}\hat{\psi}_{\tau,n}(0) \left\|\left(\mathscr{A} ,\tilde{\mathscr{A}}_{1}, \cdots, \tilde{\mathscr{A}}_{d} \right) u\right\|_2^2 \\
    =&\,\hat{\psi}_{\tau,n}(0) \sum_{k\in\Z^d, \|k\|_2\leq n} \left|\sum_{s=0}^d \hat{\mu}_s(k) \right|^2 \\
    \geq&\,\sum_{k\in\Z^d} \hat{\psi}_{\tau,n}(k) \left|\sum_{s=0}^d \hat{\mu}_s(k) \right|^2 \\
    =&\,\sum_{s, s'=0}^d \sum_{k\in\Z^d} \hat{\psi}_{\tau,n}(k) \hat{\mu}_s(k) \overline{\hat{\mu}_{s'}(k)} \\
    =&\,S_1+ S_2 + S_3 + S_4
\end{align*}
where the decomposition consists of
\begin{align*}
    S_1&= \sum_{k\in\Z^d} \hat{\psi}_{\tau,n}(k) |\hat{\mu}_0(k)|^2, \\
    S_2&= \sum_{s=1}^d 2\Re\left[\sum_{k\in\Z^d} \hat{\psi}_{\tau,n}(k) \hat{\mu}_s(k) \overline{\hat{\mu}_{0}(k)}\right], \\
    S_3&=\sum_{s=1}^d \sum_{\genfrac{}{}{0pt}{}{s'=1}{s'< s}}^d 2\Re\left[\sum_{k\in\Z^d} \hat{\psi}_{\tau,n}(k) \hat{\mu}_s(k) \overline{\hat{\mu}_{s'}(k)}\right] \text{ and} \\
    S_4&=\sum_{s=1}^d \sum_{k\in\Z^d} \hat{\psi}_{\tau,n}(k) |\hat{\mu}_s(k)|^2.
\end{align*}
By Poisson's summation formula and the separation of $Y$ together with the compact support of $\psi_{\tau,n}$ we derive
\begin{align*}
    S_1&=\sum_{t,t'\in Y} (u_0)_t \overline{(u_0)_{t'}} \sum_{k\in\Z^d} \hat{\psi}_{\tau,n}(k) \eip{(t'-t)k} \\
    &= \sum_{t\in Y} |(u_0)_t|^2 \psi_{\tau,n}(0)
\end{align*}
and analogously due to the relation between multiplication with monomials and derivatives under the Fourier transform that $S_4$ equals\footnote{Note that one can estimate $\hat{\psi}_{\tau,n}(k)\in\mathcal{O}(\|k\|_2^{-d-3})$ because $\varphi$ admits a differential equation presented in \cite{Cohn_2003}. Hence, this function allows to apply Poisson summation formula even to its second derivative.} 
\begin{align*}
    &=- \sum_{s=1}^d \sum_{t,t'\in Y}  \sum_{k\in\Z^d} \scalebox{0.78}{$(u_s)_t \overline{(u_s)_{t'}}\hat{\psi}_{\tau,n}(k) (2\pi\ii k_s)^2 \eip{(t'-t)k}$} \\
    &= - \sum_{s=1}^d \sum_{t,t'\in Y} (u_s)_t \overline{(u_s)_{t'}} \sum_{k\in\Z^d} \scalebox{0.78}{$\left(\frac{\partial^2 \psi_{\tau,n}}{\partial x_s^2}\right)\hat{}(k) \eip{(t'-t)k} $}\\
    &= - \sum_{s=1}^d \sum_{t\in Y} |(u_s)_t|^2 \left(\frac{\partial^2 \psi_{\tau,n}}{\partial x_s^2}\right) (0).
\end{align*}
Moreover, one can evaluate the cross terms $S_2$ and $S_3$ by observing
\begin{align*}
    &\hspace{0.05cm}\sum_{k\in\Z^d} \hat{\psi}_{\tau,n}(k) \hat{\mu}_s(k) \overline{\hat{\mu}_{0}(k)} \\
    =&\,\sum_{t,t'} (u_s)_t \overline{(u_0)_{t'}} \sum_{k\in\Z^d} (-2\pi\ii k_s) \hat{\psi}_{\tau,n}(k) \eip{(t'-t)k} \\
    =&\,\sum_{t} (u_s)_t \overline{(u_0)_{t}} \left(\frac{\partial \psi_{\tau,n}}{\partial x_s}\right) (0)
\end{align*}
for $s=1,\dots,d$ and
\begin{align*}
    &\hspace{0.05cm}\sum_{k\in\Z^d} \hat{\psi}_{\tau,n}(k) \hat{\mu}_s(k) \overline{\hat{\mu}_{s'}(k)} \\
    =&\,\sum_{t,t'} (u_s)_t \overline{(u_{s'})_{t'}} \sum_{k\in\Z^d} \scalebox{0.76}{$(-2\pi\ii k_s) (2\pi\ii k_{s'}) \hat{\psi}_{\tau,n}(k) \eip{(t'-t)k}$} \\
    =&\,-\sum_{t} (u_s)_t \overline{(u_{s'})_{t}} \left(\frac{\partial^2 \psi_{\tau,n}}{\partial x_s \partial x_{s'}}\right) (0)
\end{align*}
for $s,s'\in\{1,\dots,d\}, s\neq s'$. By \cref{Lem_derivatives_eval}, we have $S_2=S_3=0$ and
\begin{align*}
    &\hspace{0.05cm}\left\|\left(\mathscr{A} ,\tilde{\mathscr{A}}_{1}, \cdots, \tilde{\mathscr{A}}_{d} \right) u\right\|_2^2 \\
    \geq&\,\min\left(\frac{\psi_{\tau,n}(0)}{\hat\psi_{\tau,n}(0)},-\frac{\left(\frac{\partial^2 \psi_{\tau,n}}{\partial x_s^2}\right) (0)}{\hat\psi_{\tau,n}(0)}\right) \|u\|_2^2 \\
    \geq&\,\min\left(c_{d,\tau}^{(1)},\frac{c_d \tau  (1+\tau)^{-d/2-1} n^2}{4\pi^2(1+\tau)\hat\varphi(0)^2}\right) n^d
\end{align*}
where $c_{d,\tau}^{(1)}$ is the constant from the proof of \cite[Thm.\,3.3]{Ho_23b} and \eqref{eq:low_bound_radial} was used. Defining the constant given by the minimum as $c_{d,\tau}$ completes the proof.
\end{proof}





\bibliographystyle{IEEEtran}
\bibliography{IEEEabrv,references.bib}

\end{document}